\documentclass[11pt]{article}

\usepackage{authblk}
\usepackage{currfile}
\usepackage{floatrow}
\usepackage{amsmath,amsthm,verbatim,amssymb,amsfonts,amscd}

\usepackage{xcolor}
\usepackage{graphicx,tikz,caption,subfig}
\usepackage{mathrsfs}
\usepackage{bm}
\usepackage{float}

\usepackage{setspace}
\usetikzlibrary{arrows,shapes,positioning}
\usetikzlibrary{decorations.markings}
\tikzstyle arrowstyle=[scale=1]
\tikzstyle directed=[postaction={decorate,decoration={markings, mark=at position 0.65 with {\arrow[arrowstyle]{stealth}}}}]
\tikzstyle redirected=[postaction={decorate,decoration={markings, mark=at position 0.5 with {\arrow[arrowstyle]{stealth}}}}]

\usepackage[shortlabels]{enumitem}
\newtheorem{theorem}{Theorem}[section]

\newtheorem{lemma}[theorem]{Lemma}

\newtheorem{claim}{Claim}

\newcommand{\Z}{\mathbb{Z}}

\newcommand{\JCTB}{{\it J. Combin. Theory Ser. B}}

\newcommand{\JGT}{{\it J. Graph Theory}}

\newcommand{\DM}{{\it Discrete Math.}}

\newcommand{\SIAMDM}{{\it SIAM J. Discrete Math.}}

\newcommand{\EJC}{{\it European J. Combin.}}

\begin{document}

\title{Modulo $k$-orientations of highly edge-connected signed graphs}

\author{
You Lu\thanks{School of Mathematics and Statistics, Northwestern Polytechnical University,  Xi'an, Shaanxi 710129, China. Email: luyou@nwpu.edu.cn. Partially supported by National Natural Science Foundation of China (No. 12671418).},
~~Zhengke Miao\thanks{School of Mathematics and Statistics \& Key Laboratory of Analytical Mathematics and Applications (Ministry of Education), Fujian Normal University, Fuzhou, Fujian 350007, China. Email: zkmiao@jsnu.edu.cn. Partially supported by National Natural Science Foundation of China (No. 12431013) and the National Key R\&D Program of China (No. 2024YFA1013900).},
~~Zhenyu Tian\thanks{School of Mathematics and Statistics, Northwestern Polytechnical University,  Xi'an, Shaanxi 710129, China. Email: tianzhenyu@mail.nwpu.edu.cn.}, 
~~Cun-Quan Zhang\thanks{Department of Mathematics, West Virginia University, Morgantown, WV 26506, USA. Email:  cqzhang@mail.wvu.edu.}
}

\date{}
\maketitle

\begin{abstract}
A modulo $k$-orientation of a signed graph is an orientation in which the outdegree and indegree are congruent modulo $k$ at every vertex. We show that, for every odd integer $k\ge3$, every essentially $k$-unbalanced $(6k-10)$-edge-connected signed graph admits a modulo $k$-orientation. This improves the edge-connectivity bound of Zhu [\JCTB, 112 (2015): 93--103] from $6k-7$ to $6k-10$.

\medskip
\noindent \textbf{Keywords:} {\small signed graph, modulo orientation, edge connectivity}
\end{abstract}

\section{Introduction}

All graphs considered in this paper are finite and may have parallel edges and loops. For notation and terminology not defined here, we follow Bondy and Murty~\cite{BM2008} and Zhang~\cite{Zhang1997}.

A \emph{signed graph} $(G,\sigma)$ is a graph $G$ together with a {\em signature} $\sigma\colon E(G)\to\{1,-1\}$. We write $G$ for $(G,\sigma)$ when the signature is clear, and use $\sigma_G$ when necessary. An edge $e$ is \emph{positive} if $\sigma(e)=1$ and \emph{negative} otherwise. A signed graph with no negative edges is called an {\em ordinary} graph. Let $G^+$ and $G^-$ denote the spanning subgraphs of $G$ consisting of all positive and all negative edges, respectively. \emph{Switching} at a vertex reverses the signs of its incident non-loop edges. Two signed graphs are \emph{equivalent} if one can be obtained from the other by a sequence of switchings. The \emph{negativeness} of $G$ is
$$
\epsilon(G)=\min\{|E(G_1^-)|:G_1\text{ is equivalent to }G\}.
$$

Every edge of $G$ consists of two half-edges, one incident with each end.  
An \emph{orientation} $D$ of $G$ assigns a direction to each half-edge such that a positive edge has one half-edge directed away from its end and the other directed towards its end, as shown in 
\begin{tikzpicture}[scale=0.8]
\path (0.8,0) coordinate (2);\draw [fill=black] (2) circle (0.08cm);
\path (-0.8,0) coordinate (1);\draw [fill=black] (1) circle (0.08cm);
\draw[directed, line width=0.85pt] (1) -- (0,0);
\draw[directed,  line width=0.85pt] (0,0) -- (2);
\end{tikzpicture}
 or
  \begin{tikzpicture}[scale=0.8]
\path (0.8,0) coordinate (1);\draw [fill=black] (1) circle (0.08cm);
\path (-0.8,0) coordinate (2);\draw [fill=black] (2) circle (0.08cm);
\draw[directed,  line width=0.85pt] (1) -- (0,0);
\draw[directed, line width=0.85pt] (0,0) -- (2);
\end{tikzpicture}, whereas a negative edge has both half-edges directed either away from or toward their ends, as shown in 
 \begin{tikzpicture}[scale=0.8]
\path (0.8,0) coordinate (2);\draw [fill=black] (2) circle (0.08cm);
\path (-0.8,0) coordinate (1);\draw [fill=black] (1) circle (0.08cm);
\draw[directed, dotted, line width=0.85pt] (1) -- (0,0);
\draw[directed, dotted, line width=0.85pt] (2) -- (0,0);
\end{tikzpicture}
 or
 \begin{tikzpicture}[scale=0.8]
\path (0.8,0) coordinate (2);\draw [fill=black] (2) circle (0.08cm);
\path (-0.8,0) coordinate (1);\draw [fill=black] (1) circle (0.08cm);
\draw[directed, dotted, line width=0.85pt] (0,0)--(1);
\draw[directed, dotted, line width=0.85pt]  (0,0)--(2);
\end{tikzpicture}. 
Such a negative edge is called a \emph{sink edge} in the former case and a \emph{source edge} in the latter.
The \emph{outdegree} $d_D^+(v)$ and \emph{indegree} $d_D^-(v)$ are the numbers of half-edges directed away from and toward $v$, respectively.  The \emph{boundary} of $D$ at $v$ is the integer
$$
\partial_D(v)=d_D^+(v)-d_D^-(v).
$$
Let $k$ be a positive integer. An orientation $D$ is called a \emph{modulo $k$-orientation} of $G$ if
$$
 \partial_D(v)\equiv0 \pmod{k}\quad \forall v\in V(G).
$$
Clearly, the existence of a modulo $k$-orientation is invariant under switching.

The subject of modulo $k$-orientations is motivated by and closely related to several mainstream topics in graph theory, including Tutte's integer flow theory and related flow problems \cite{DKKS10,FLXZZ2008,Jaeger79,Seymour81,Steffen10}, Jaeger's modulo orientation problems \cite{HLWZ,Jaeger84,JLPT,Thomassen12}, circular coloring and flow problems \cite{EMT16,PZ03,RZ,Steffen15,Zhu2001}, and Bouchet's $6$-flow conjecture and related flow problems for signed graphs \cite{Bouchet83,DLLLZZ,DNS25,HLLL25,MS15,Nurse26}.

 For even $k\ge2$, an ordinary graph admits a modulo $k$-orientation if and only if every vertex has even degree. For odd $k\ge 3$, Lov\'asz et al.~\cite{LTWZ} showed that sufficiently high edge connectivity guarantees a modulo $k$-orientation.

\begin{theorem}[\cite{LTWZ}]
\label{thm: LTWZ}
Let $k\ge3$ be an odd integer. Every $(3k-3)$-edge-connected graph admits a modulo $k$-orientation.
\end{theorem}

For signed graphs, high edge connectivity alone does not guarantee the existence of a modulo $k$-orientation. Let $k\ge3$ be an odd integer. If $|E(G^-)|$ is odd and less than $k$, then every orientation $D$ of $G$ satisfies
$$
\sum_{v\in V(G)} \partial_D(v)=2(2s-|E(G^-)|)\not\equiv0\pmod{k},
$$
where $s$ is the number of sink edges.  Motivated by this, Zhu~\cite{Zhu} defined a signed graph to be \emph{essentially $k$-unbalanced} if every equivalent signed graph has either an even number of negative edges or at least $k$ negative edges. Being essentially $k$-unbalanced is a necessary condition for the existence of a modulo $k$-orientation.
 
Theorems \ref{thm: WYZZ} and \ref{thm:Zhu} below show that essential $k$-unbalancedness is sufficient when the edge connectivity is sufficiently high.
   
\begin{theorem}[\cite{WYZZ}]
\label{thm: WYZZ}
Every essentially $3$-unbalanced $8$-edge-connected signed graph admits a modulo $3$-orientation.
\end{theorem}

\begin{theorem}[\cite{Zhu}]
\label{thm:Zhu}
Let $k\ge3$ be an odd integer. Every essentially $k$-unbalanced $(6k-7)$-edge-connected signed graph admits a modulo $k$-orientation.
\end{theorem}

In this paper, we generalize Theorem~\ref{thm: WYZZ} to all odd integers $k\ge 3$ and improve the edge-connectivity bound in Theorem \ref{thm:Zhu} as follows.

\begin{theorem}
\label{thm: main}
Let $k\ge3$ be an odd integer. Every essentially $k$-unbalanced $(6k-10)$-edge-connected signed graph admits a modulo $k$-orientation.
\end{theorem}

The rest of this paper is organized as follows. Section~\ref{sec: prelim} introduces the notation and presents results on splitting, contraction and orientations. Section~\ref{sec: selection} establishes a lemma on subgraphs with modular constraints. Section~\ref{sec: proof} uses this lemma to prove Theorem~\ref{thm: main}.


\section{Preliminaries}
\label{sec: prelim}

Let $G$ be a signed graph. For $A,B\subseteq V(G)$, let $e_G(A,B)$ be the number of edges with one end in $A$ and the other in $B$. In particular, $e_G(A,A)=|E(G[A])|$. For simplicity, we write $e_G(u,B)$, $e_G(A,v)$, and $e_G(u,v)$ for $e_G(\{u\},B)$, $e_G(A,\{v\})$, and $e_G(\{u\},\{v\})$, respectively. Let $\delta_G(A)$ be the edge cut between $A$ and $V(G)\setminus A$, and put $d_G(A)=|\delta_G(A)|$. Write $\delta_G(v)=\delta_G(\{v\})$, and let $d_G(v)$ be the degree of $v$, with each loop counted twice. Thus $d_G(v)=d_G(\{v\})$ when $v$ has no loop. Let $V_d(G)=\{v\in V(G):d_G(v)=d\}$.

The \emph{local edge connectivity} between distinct vertices $u$ and $v$ is defined by
$$
\lambda_G(u,v)=\min\{d_G(A):u\in A\subseteq V(G)\setminus\{v\}\}.
$$
We say that $G$ is {\em $k$-edge-connected } if $\lambda_G(u,v)\ge k$ for all distinct vertices $u, v\in V(G)$.

For distinct edges $uv_1, uv_2\in E(G)$, let $G_{[uv_1, uv_2]}$ be the signed graph obtained from $G$ by deleting $uv_1$ and $uv_2$ and adding a new edge $v_1v_2$ with sign $\sigma_G(uv_1)\sigma_G(uv_2)$. This operation is called {\em splitting} $uv_1$ and $uv_2$ at $u$ (see Fig.~\ref{fig: splitting}). 
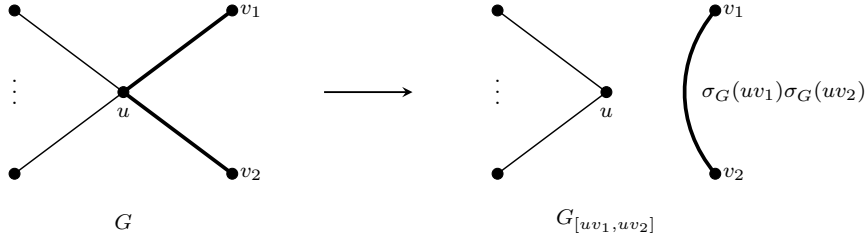
\begin{figure}[!t]
\centering
\scriptsize
\begin{tikzpicture}[scale=0.9]

\path (-3.3,0) coordinate (u);
\path (-4.9,1.2) coordinate (a);
\path (-4.9,-1.2) coordinate (b);
\path (-1.7,1.2) coordinate (v1);
\path (-1.7,-1.2) coordinate (v2);

\draw[line width=0.55pt] (a) -- (u) -- (b);
\draw[line width=1.4pt] (u) -- (v1);
\draw[line width=1.4pt] (u) -- (v2);
\node at (-4.9,0.12) {$\vdots$};

\foreach \v in {u,a,b,v1,v2}
 \draw[fill=black] (\v) circle (0.08cm);
\node[below=3pt] at (u) {$u$};
\node[right] at (v1) {$v_1$};
\node[right] at (v2) {$v_2$};
\node at (-3.3,-1.9) {$G$};

\draw[->,>=stealth,line width=0.7pt] (-0.35,0) -- (0.95,0);

\path (3.8,0) coordinate (up);
\path (2.2,1.2) coordinate (ap);
\path (2.2,-1.2) coordinate (bp);
\path (5.4,1.2) coordinate (v1p);
\path (5.4,-1.2) coordinate (v2p);

\draw[line width=0.55pt] (ap) -- (up) -- (bp);
\draw[line width=1.4pt] (v1p) to[bend right=40]
  node[midway,right=3pt] {$\sigma_G(uv_1)\sigma_G(uv_2)$} (v2p);
\node at (2.2,0.12) {$\vdots$};

\foreach \v in {up,ap,bp,v1p,v2p}
  \draw[fill=black] (\v) circle (0.08cm);
\node[below=3pt] at (up) {$u$};
\node[right] at (v1p) {$v_1$};
\node[right] at (v2p) {$v_2$};
\node at (3.8,-1.9) {$G_{[uv_1,uv_2]}$};
\end{tikzpicture}
\caption{Splitting $uv_1$ and $uv_2$ at $u$.}
\label{fig: splitting}
\end{figure}
Observe that every modulo $k$-orientation of $G_{[uv_1, uv_2]}$ can be lifted to a modulo $k$-orientation of $G$.

\begin{lemma}[\cite{WYZZ}]
\label{lem:splitting-negativeness}

Let $G$ be a signed graph and let $uv_1$ and $uv_2$ be two distinct edges of $G$. Then $\epsilon(G_{[uv_1,uv_2]})\ge\epsilon(G)-2$.

\end{lemma}

\begin{lemma}[\cite{Mader}]
\label{lem:Mader}
Let $G$ be a connected graph and let $u\in V(G)$ such that $|\delta_G(u)|\ge4$ and no edge in $ \delta_G(u)$ is a cut-edge. Then $G$ contains two edges $uv_1$ and $uv_2$ such that $\lambda_{G_{[uv_1, uv_2]}}(x,y)=\lambda_G(x,y)$ for any distinct $x,y\in V(G)\setminus\{u\}$.
\end{lemma}

For a nonempty subset $A\subseteq V(G)$, let $G/A$ denote the signed graph obtained from $G$ by identifying all vertices of $A$ into a single vertex and deleting all resulting positive loops. This operation is called \emph{contracting} $A$ (see Fig.~\ref{fig: contraction}).
\begin{figure}[!t]
\centering
\scriptsize
\begin{tikzpicture}[scale=0.9]
\draw[black!45,line width=0.5pt] (-5.3,0) ellipse (1.45cm and 1.2cm);
\draw[black!45,line width=0.5pt] (-1.65,0) ellipse (1.05cm and 1.2cm);

\path (-6.05,0.55) coordinate (a1);
\path (-5.35,0.55) coordinate (a2);
\path (-6.05,-0.55) coordinate (a3);
\path (-5.35,-0.55) coordinate (a4);
\path (-4.35,0.55) coordinate (c1);
\path (-4.35,-0.55) coordinate (cs);
\path (-2.2,0.55) coordinate (b1);
\path (-2.2,-0.55) coordinate (bs);

\draw[dotted,line width=0.85pt] (a1) -- node[above] {$e_1$} (a2);
\draw[dotted,line width=0.85pt] (a3) -- node[below] {$e_r$} (a4);
\node at (-5.7,0.12) {$\vdots$};
\draw[line width=0.85pt] (c1) -- node[above] {} (b1);
\draw[line width=0.85pt] (cs) -- node[below] {} (bs);
\node at (-3.275,0.12) {$\vdots$};

\foreach \v in {a1,a2,a3,a4,c1,cs,b1,bs}
  \draw[fill=black] (\v) circle (0.08cm);
\node at (-5.3,-1.5) {$A$};
\node at (-1.65,-1.5) {$V(G)\setminus A$};
\node at (-3.7,-2.1) {$G$};

\draw[->,>=stealth,line width=0.7pt] (0.1,0) -- (1.4,0);

\draw[black!45,line width=0.5pt] (5.6,0) ellipse (1.05cm and 1.2cm);
\path (3.0,0) coordinate (z);
\path (5.05,0.55) coordinate (b1p);
\path (5.05,-0.55) coordinate (bsp);

\draw[line width=0.85pt] (z) -- node[midway,above] {} (b1p);
\draw[line width=0.85pt] (z) -- node[midway,below] {} (bsp);
\node at (5.05,0.12) {$\vdots$};
\draw[dotted,line width=0.85pt]
  (z) .. controls (1.9,1.4) and (4.1,1.4) ..
  node[midway,above] {$e_1$} (z);
\draw[dotted,line width=0.85pt]
  (z) .. controls (1.9,-1.4) and (4.1,-1.4) ..
  node[midway,below] {$e_r$} (z);
\node at (2.4,0.12) {$\vdots$};

\foreach \v in {z,b1p,bsp}
  \draw[fill=black] (\v) circle (0.08cm);
\node at (5.6,-1.5) {$V(G)\setminus A$};
\node at (4.3,-2.1) {$G/A$};
\end{tikzpicture}
\caption{Contracting $A$. Dotted edges are negative.}
\label{fig: contraction}
\end{figure}
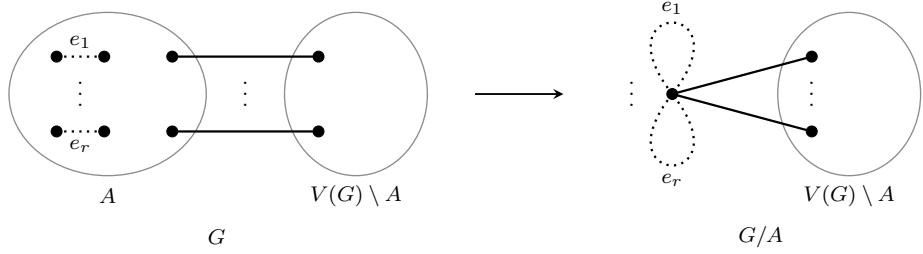

\begin{lemma}[\cite{WYZZ}]\label{le: contract}
Let $G$ be a signed graph and let $A\subseteq V(G)$. Then $\epsilon(G/A)\ge\epsilon(G)$. If $|E(G^-)|=\epsilon(G)$, then $\epsilon(G/A)=\epsilon(G)$.
\end{lemma}

Let $k\ge3$ be an odd integer, and let $G$ be a graph. A mapping $\beta: V(G)\to\Z_k$ is called a \emph{$\Z_k$-boundary} of $G$ if $\sum_{v\in V(G)}\beta(v)\equiv 0 \pmod k$.  An orientation $D$ of $G$ is a \emph{$\beta$-orientation} if
$$
 \partial_D(v)\equiv\beta(v)\pmod{k}
\quad \forall v\in V(G).
$$
For $A\subseteq V(G)$, let $\beta(A)= \sum_{v\in A}\beta(v)$ in $\Z_k$, and choose $\tau(A)\in[-k,k]$ such that
$$
\tau(A)\equiv\beta(A)\pmod{k},\qquad
\tau(A)\equiv d_G(A)\pmod{2}.
$$
These congruences determine $\tau(A)$ modulo $2k$, so $|\tau(A)|$ is uniquely determined. Write $\tau(v)=\tau(\{v\})$. Since $|\tau(A)|\le k$ and $d_G(A)-|\tau(A)|$ is even,
we will use the following inequality:
\begin{equation}\label{eq: parity bound}
d_G(A)\ge2k-2+|\tau(A)| \mbox{ \ if }d_G(A)\ge3k-3.
\end{equation}

\begin{lemma}[\cite{LTWZ}]\label{thm: extension}
Let $k\ge 3$ be an odd integer, let $G$ be a loopless ordinary graph with a $\Z_k$-boundary $\beta$ and $z\in V(G)$. Let $D_z$ be a pre-orientation of $\delta_G(z)$. Assume that 
\begin{itemize}
\item[\rm (i)] $|V(G)|\ge 3$;
\item[\rm (ii)] $d_G(z)\leq 2k-2+|\tau(z)|$ and $\partial_{D_z}(z)\equiv \beta(z) \pmod k$;
\item[\rm (iii)] $d_G(A)\ge 2k-2+|\tau(A)|$ for each nonempty $A\subseteq V(G)\setminus\{z\}$ with $|V(G)\setminus A|\ge2$. 
\end{itemize}
Then $D_z$ can be extended to a $\beta$-orientation $D$ of $G$.
\end{lemma}

\begin{lemma}\label{lem:few}
Let $k\ge3$ be an odd integer and let $G$ be a $(3k-3)$-edge-connected signed graph with at least two vertices.
If $|E(G^-)|$ is even and at most $k-1$, or if $|E(G^-)|=k$, then $G$ admits a modulo $k$-orientation.
\end{lemma}

\begin{proof}
Without loss of generality, assume that $G$ contains no positive loop. Let $H$ be the ordinary graph obtained from $G^+$ by adding a new vertex $z$ and, for each $uv\in E(G^-)$, adding two edges $uz$ and $vz$. Then $H$ is loopless and $|V(H)|\ge 3$.

We define a new graph $H'$ and a pre-orientation $D_z$ at $z$ as follows.
If $|E(G^-)|$ is even, let $H'=H$ and choose half the edges in $E(G^-)$.
For each $uv\in E(G^-)$, direct both $uz$ and $vz$ toward $z$
if $uv$ is chosen, and away from $z$ otherwise.
Then $d_{H'}(z)=2|E(G^-)|\le2k-2$ and $\partial_{D_z}(z)=0$.
If $|E(G^-)|=k$, initially direct all edges in $\delta_H(z)$ toward $z$,
and then replace one edge $xz$ with $k-1$ parallel edges directed
from $z$ to $x$. Denote the resulting graph by $H'$.
Then $d_{H'}(z)=3k-2$ and $\partial_{D_z}(z)=-k$.

Take $\beta\equiv0$.
Then $d_{H'}(z)\le2k-2+|\tau(z)|$ and
$\partial_{D_z}(z)\equiv\beta(z)\pmod{k}$.
Moreover, for every nonempty $A\subseteq V(H')\setminus\{z\}$
with $|V(H')\setminus A|\ge2$, we have $d_{H'}(A)\ge d_G(A)\ge3k-3$.
By applying \eqref{eq: parity bound} in $H'$, we obtain
$$
d_{H'}(A)\ge2k-2+|\tau(A)|.
$$
Therefore, by Lemma~\ref{thm: extension}, $D_z$ can be extended to a modulo $k$-orientation of $H'$.

If $|E(G^-)|=k$, we first replace the $k-1$ parallel edges
directed from $z$ to $x$ with a single edge $xz$ directed
toward $z$.
In either case, we then replace each pair $uz,vz$ with
the corresponding negative edge $uv$, keeping the half-edge
directions at $u$ and $v$ unchanged.
After deleting $z$, we obtain a modulo $k$-orientation of $G$.
\end{proof}
\section{Subgraphs with modular constraints}
\label{sec: selection}

In this section, we establish Lemma~\ref{lem: star}, which plays an important role in the proof of Theorem \ref{thm: strengthened}.

For integers $a\le b$, set $[a,b]=\{a,a+1,\dots,b\}$. 
 For an integer $s$ and a positive integer $k$, let $\langle s\rangle_k$ denote the unique
integer in $[0,k-1]$ congruent to $s$ modulo $k$.
 For nonempty subsets $X_1,\dots,X_r$ of an abelian group, let
$
X_1+\cdots+X_r
=\{x_1+\cdots+x_r:x_i\in X_i, \ i\in[1,r]\}.
$

\begin{lemma}[\cite{Kneser1953}, see also~\cite{DeVos2014}]\label{lem: Kneser}
Let $X,Y$ be finite nonempty subsets of an abelian group $\Gamma$, and let $K=\{z\in\Gamma:(X+Y)+z=X+Y\}$. Then $K$ is a finite subgroup of $\Gamma$ and
$$
|X+Y|\ge|X+K|+|Y+K|-|K|.
$$
\end{lemma}

\begin{lemma}\label{lem: additive}
Let $k\ge5$ be an odd integer, let $a\in[0,k-1]$, and let $X,Y$ be nonempty subsets of $\mathbb{Z}_k$.
If $X=[0,a]$ or $X=[0,a]\setminus\{r\}$ for some $r\in[0,a]$, then
$$
|X+Y|\ge\min\{k,\,|X|+|Y|-1\}.
$$
\end{lemma}

\begin{proof}
For every nonempty proper subset $S\subset \Z_k$, we have $|S+\{0,1\}|=|S\cup(S+1)|\ge|S|+1$.
If $X=[0,a]$, then we repeatedly apply this inequality to obtain the desired bound. Therefore,  assume that $X=[0,a]\setminus\{r\}$ and $X+Y\ne\Z_k$. 

Let $K=\{z\in\Z_k:(X+Y)+z=X+Y\}$. By Lemma~\ref{lem: Kneser}, $|X+Y|\ge|X+K|+|Y+K|-|K|$,
which gives the desired bound if $K=\{0\}$. We may therefore assume that $K\ne\{0\}$.

Let $q=|K|$ and $\ell=k/q$. Since $X+Y\ne\Z_k$, the subgroup $K$ is proper, and so $q,\ell\ge3$. The cosets of $K$ are the residue classes modulo $\ell$. Moreover, since $Y$ is nonempty and $X+Y+K=X+Y\ne\Z_k$, we have $X+K\ne\Z_k$.

If $a\le\ell-1$, then the elements of $X$ lie in distinct cosets of $K$, and hence $|X+K|-q=q(a-1)\ge a-1$.  Suppose now that $a\ge\ell$. Then $[0,a]$ meets every coset of $K$, and so $X$ meets at least $\ell-1$ cosets. If $a\ge2\ell-1$, each coset contains at least two elements
of $[0,a]$ and hence still meets $X$ after $r$ is removed, contradicting $X+K\ne\Z_k$.
Thus $a\le2\ell-2$ and $|X+K|-q\ge q(\ell-2)\ge3(\ell-2)\ge2\ell-3\ge a-1$.

In either case, since $|Y+K|\ge|Y|$, by Lemma~\ref{lem: Kneser}, we have $|X+Y|\ge a+|Y|-1=|X|+|Y|-1$, as required.
\end{proof}

\begin{lemma}[{\cite{FLS}}]\label{thm: avoidance}
Let $G$ be a loopless graph, let $B\subseteq V(G)$, and let
$L_v\subseteq[0,d_G(v)]$ for each $v\in B$.
If
$$
\sum_{v\in A}|L_v|\le e_G(A,A)+d_G(A)\quad \forall A\subseteq B,
$$
then $G$ has a spanning subgraph $H$ such that
$d_H(v)\notin L_v$ for every $v\in B$.
\end{lemma}

\begin{lemma}\label{lem: star}
Let $k\ge5$ be an odd integer, and let $G$ be a loopless graph
with maximum degree $\Delta=3k-5$ and multiplicity at most $k-2$. Let $A\subseteq V(G)$ satisfy $d_G(A)\le 3k-4$, and set 
$$C=A\cap V_\Delta(G), \qquad b_v=e_G(v,V(G)\setminus A) \quad \forall v\in A.$$
Suppose that $|C|\ge3$, and let $u,w\in C$ be distinct vertices
such that $b_u\le b_v\le b_w$ for every $v\in C$.
Then, for any $s\in\Z_k$ and any choice of residues $f_v\in\Z_k$ for $v\in C\setminus\{u,w\}$,
there exists a spanning subgraph $H$ of $G[A]$ satisfying
$$
|E(H)|\equiv s\pmod{k},
\qquad
d_H(v)\not\equiv f_v\pmod{k}
\quad \forall v\in C\setminus\{u,w\}.
$$
\end{lemma}

\begin{proof}
By the choice of $u$, we have $|C|b_u\le\sum_{v\in C}b_v\le d_G(A)\le \Delta+1=3k-4$.
Since $|C|\ge 3$, it follows that $b_u\le k-2$ and $d_{G[A]}(u)=\Delta-b_u\ge2k-3$.
For $v\in C\setminus \{u, w\}$, the choice of $w$ yields $2b_v\le b_v+b_w\le d_G(A)\le\Delta+1$, and so $b_v\le \frac{1}{2}\Delta$ since $\Delta$ is even.
Let
$$
W_1=\{v\in C\setminus \{u, w\} : e_G(u,v)=1\},
\qquad
W_2=\{v\in C\setminus \{u, w\} : e_G(u,v)\ge2\}.
$$

\begin{figure}[!t]
\scriptsize
\begin{center}
\begin{tikzpicture}[scale=0.8]
\path (0,3.3) coordinate (u);
\path (-4.5,3.0) coordinate (w);
\path (-4.5,0) coordinate (v2);
\path (-1.5,0) coordinate (vb);
\path (1.5,0) coordinate (v1);
\path (4.05,0) coordinate (va);
\path (4.95,0) coordinate (va2);

\draw[black!45, line width=0.5pt]
  (-4.5,0) ellipse (1.05cm and 0.45cm);
\draw[black!65, fill=black!6, line width=0.7pt]
  (-1.5,0) ellipse (1.05cm and 0.45cm);
\draw[black!45, line width=0.5pt]
  (1.5,0) ellipse (1.05cm and 0.45cm);
\draw[black!45, line width=0.5pt]
  (4.5,0) ellipse (1.05cm and 0.45cm);

\draw[line width=0.55pt] (u) -- (w);
\draw[line width=0.55pt] (u) to[bend right=9] (v2);
\draw[line width=0.55pt] (u) to[bend left=9] (v2);
\draw[line width=1.4pt] (u) -- (vb);
\draw[line width=0.55pt] (u) -- (v1);
\draw[line width=0.55pt] (u) -- (va);
\draw[line width=0.55pt] (u) -- (va2);
\draw[line width=0.55pt] (u) to[bend left=12] (va2);

\foreach \v in {u, w, v2, vb, v1, va, va2}
  \draw[fill=black] (\v) circle (0.08cm);

\node[above] at (u) {$u$};
\node[above] at (w) {$w$};
\node[below] at (-4.5,-0.48) {$W_2$};
\node[below] at (-1.5,-0.48) {$B$};
\node[below] at (1.5,-0.48) {$W_1\setminus B$};
\node[below] at (4.5,-0.48) {$A\setminus C$};

\end{tikzpicture}
\end{center}
\caption{The set $B\subseteq W_1$.}
\label{fig: select B}
\end{figure}
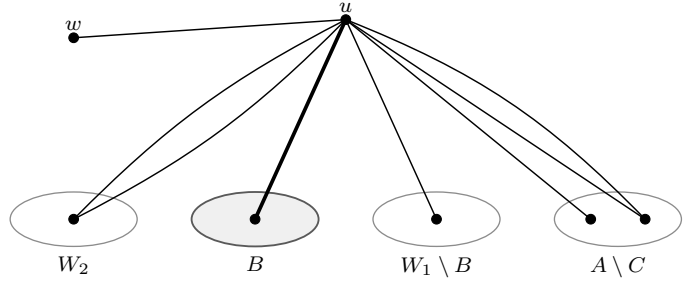
\begin{claim}\label{cl: select B}
There exists a subset $B\subseteq W_1$ (see Fig.~\ref{fig: select B}) such that
\begin{equation*}
\rho =d_{G[A]}(u)-|W_2|-|W_1\setminus B|\ge k-1.
\end{equation*}
Moreover, when $k=5$, $b_v\le1$ for every $v\in B$, with at most one exception for which $b_v=2$. 
\end{claim}

\begin{proof}[Proof of Claim \ref{cl: select B}.] If $k\ge7$, take $B=W_1$. Since $2|W_2|\le d_{G[A]}(u)$, we have
$$
\rho=d_{G[A]}(u)-|W_2|
\ge\left\lceil \frac{1}{2}d_{G[A]}(u)\right\rceil
\ge k-1.
$$
We now assume that $k=5$, and initially put $B=\{v\in W_1:b_v\le1\}$. If $\rho\ge4$, this choice has the required properties. We  assume further that $\rho\le3$.

Since $d_{G[A]}(u)\ge 2k-3=7$, we have $|C|\ge |W_1|+|W_2|+2\ge d_{G[A]}(u)-\rho+2\ge 6,$
which implies that $b_u\le \lfloor (3k-4)/|C| \rfloor\le1$ and $d_{G[A]}(u)=\Delta-b_u\ge9$.

Note that every vertex $v\in W_1\setminus B$ has $b_v\ge2$. If $|W_1\setminus B|\ge5$, then  $b_w\ge2$ and hence
$d_G(A)\ge\sum_{v\in W_1\setminus B}b_v+b_w
\ge2(|W_1\setminus B|+1)\ge12$,
contrary to $d_G(A)\le11$. Thus $|W_1\setminus B|\le4$.
Since $d_{G[A]}(u)\ge|W_1\setminus B|+2|W_2|$, we have $|W_2|\le\rho$ and 
$$
4\ge|W_1\setminus B|
=d_{G[A]}(u)-\rho-|W_2|
\ge9-2\rho
\ge3,
$$
which implies $\rho=3$ and $|W_1\setminus B|\ge3$. Furthermore, at least one vertex $v\in W_1\setminus B$ satisfies $b_v=2$, for otherwise $b_w\ge3$ and 
$
d_G(A)\ge\sum_{v\in W_1\setminus B}b_v+b_w
\ge3(|W_1\setminus B|+1)\ge12,
$
again a contradiction. Adding such a vertex to $B$ decreases $|W_1\setminus B|$
by one and therefore increases $\rho$ from $3$ to $4$.
Thus, the enlarged set $B$ satisfies the required properties. The claim holds.
\end{proof}

Let $Q$ be the spanning subgraph of $G[A]$ consisting of all
edges incident with $u$ except those joining $u$ to $W_1\setminus B$.  
Then $d_Q(u)=|W_2|+\rho$. Let $F=G[A]-E(Q)$ and for $v\in C\setminus(W_2\cup\{u,w\})$, let $t_v=d_F(v)=\Delta-b_v-d_Q(v)$. 
Then $d_Q(v)=1$ if $v\in B$ and $d_Q(v)=0$ otherwise.
Define
$$
L_v=\{j\in[0,t_v]:
j\equiv f_v\text{ or }f_v-d_Q(v)\pmod{k}\}.
$$

\begin{claim}\label{cl: Lv}
There exists a spanning subgraph $H'$ of $F$ such that
$d_{H'}(v)\notin L_v$ for every
$v\in C\setminus(W_2\cup\{u,w\})$.
\end{claim}

\begin{proof}[Proof of Claim~\ref{cl: Lv}]
We first prove that $2|L_v|\le t_v$ for every
$v\in C\setminus(W_2\cup\{u,w\})$, with at most one exception
satisfying $2|L_v|\le t_v+1$.
If $d_Q(v)=0$, then $t_v=\Delta-b_v\ge \frac{1}{2}\Delta\ge5$ and 
$$
|L_v|
\le\lfloor \frac{1}{k}t_v\rfloor+1
\le \frac{1}{5}t_v+1
\le \frac{1}{2}t_v.
$$

If $d_Q(v)=1$ and $k\ge7$, then
$t_v=\Delta-b_v-1\ge \frac{1}{2}\Delta-1\ge7$.
When $t_v=7$, we have $k=7$ and $|L_v|\le3$ since only the residue $0$ has two representatives in $[0,7]$.
When $t_v\in\{8,9\}$, each forbidden residue has at most
two representatives and thus $|L_v|\le4\le \frac{1}{2}t_v$.
When $t_v\ge10$, we have
$$
|L_v| \le2 \left\lfloor \frac{1}{k}t_v\right\rfloor+2 \le \frac{2}{7}t_v+2 \le \frac{1}{2}t_v.
$$

If $d_Q(v)=1$ and $k=5$, then $v\in B$ and $t_v=9-b_v$.
By Claim \ref{cl: select B}, we have $t_v\in\{8,9\}$,
except possibly at one vertex where $t_v=7$.
Since each residue class modulo $5$ has at most two
representatives in $[0,9]$, we have $|L_v|\le4$.
Thus $2|L_v|\le t_v$ when $t_v\in\{8,9\}$, while
$2|L_v|\le t_v+1$ at the possible exceptional vertex.
This proves the required bounds on $|L_v|$.

Since at most one vertex is exceptional, for every
$S\subseteq C\setminus(W_2\cup\{u,w\})$, we have
$$
\sum_{v\in S}|L_v|
\le\left\lfloor\frac{1}{2}\left(\sum_{v\in S}t_v+1\right)\right\rfloor
=e_F(S,S)+\left\lceil\frac{1}{2}d_F(S)\right\rceil
\le e_F(S,S)+d_F(S).
$$
Here the equality follows from
$\sum_{v\in S}t_v=2e_F(S,S)+d_F(S)$.
By Lemma~\ref{thm: avoidance}, $F$ has 
a spanning subgraph $H'$ with the required property.
\end{proof}

Fix such a subgraph $H'$.  For each $v\in A\setminus\{u\}$, let
$$
Y_v=
\begin{cases}
[0, e_G(u,v)]\setminus\{\langle f_v-d_{H'}(v)\rangle_k\},
&v\in W_2,\\
[0,d_Q(v)],
&v\notin W_2.
\end{cases}
$$
Since $e_G(u,v)\le k-2$, each $Y_v$ is a nonempty interval in $[0,k-1]$ with at most one element removed. Moreover, Claim~\ref{cl: select B} gives
$
\sum_{v\in A\setminus\{u\}}(|Y_v|-1)
\ge d_Q(u)-|W_2|
=\rho
\ge k-1.
$
Repeatedly applying Lemma~\ref{lem: additive}, we obtain $\sum_{v\in A\setminus\{u\}}Y_v=\Z_k$.
Therefore, we can choose $x_v\in Y_v$ such that
$$
\sum_{v\in A\setminus\{u\}}x_v
\equiv s-|E(H')|\pmod{k}.
$$

For each $v\in A\setminus\{u\}$, add $x_v$ edges of $Q$ joining $u$ and $v$ to $H'$, and denote the resulting spanning subgraph by $H$. Then $|E(H)|\equiv s\pmod{k}$. For $v\in W_2$, the definition of $Y_v$ excludes the residue
$\langle f_v-d_{H'}(v)\rangle_k$, and so $d_H(v)=d_{H'}(v)+x_v\not\equiv f_v\pmod{k}$.
For $v\in C\setminus(W_2\cup\{u,w\})$, we have
$x_v\in\{0,d_Q(v)\}$.
Since $d_{H'}(v)\notin L_v$ by Claim \ref{cl: Lv}, neither $d_{H'}(v)$ nor
$d_{H'}(v)+d_Q(v)$ is congruent to $f_v$ modulo $k$,
and hence $d_H(v)=d_{H'}(v)+x_v\not\equiv f_v\pmod{k}$.
Thus $H$ satisfies all the required conditions of the lemma.
\end{proof}


\section{Proof of Theorem~\ref{thm: main}}
\label{sec: proof}

In this section, we complete the proof of Theorem \ref{thm: main}. The case $k=3$ follows from  Theorem \ref{thm: WYZZ}. Thus, we pay attention to the remaining case $k\ge 5$. We prove a stronger statement, which allows one vertex to have smaller even degree.

\begin{theorem}\label{thm: strengthened}
Let $k\ge5$ be an odd integer and let $G$ be an essentially
$k$-unbalanced signed graph.
Let $V_0=\emptyset$ or $V_0=\{v_0\}$, where $v_0$ is not incident with any loop and has even degree at most $6k-12$.
If $|V(G)\setminus V_0|\ge 2$ and 
\begin{equation}\label{eq: local conn}
\lambda_G(x,y)\ge6k-10
\quad\text{for all distinct }x,y\in V(G)\setminus V_0,
\end{equation}
then $G$ admits a modulo $k$-orientation.
\end{theorem}

Taking $V_0=\emptyset$ in Theorem~\ref{thm: strengthened} yields Theorem~\ref{thm: main} for $k\ge5$ and $|V(G)|\ge2$. The case $|V(G)|\le1$ follows directly from essential $k$-unbalancedness.

\begin{proof}[Proof of Theorem~\ref{thm: strengthened}]
Suppose, to the contrary, that $G$ is a counterexample to Theorem \ref{thm: strengthened} with $|V(G)|+|E(G)|$ minimum.
Then $G$ has no positive loop and $v_0$, if present, is not isolated.  Hence $G$ is connected.
By switching if necessary, we may assume that $|E(G^-)|=\epsilon(G)$.
This implies that $d_{G^+}(B)\ge d_{G^-}(B)$ for every $B\subseteq V(G)$.

\begin{claim}\label{cl: negativeness}
$\epsilon(G)\ge k+1$.
\end{claim}

\begin{proof}[Proof of Claim~\ref{cl: negativeness}.]
Suppose $\epsilon(G)\le k$.
Since $G$ is essentially $k$-unbalanced, either $\epsilon(G)$ is even
or $\epsilon(G)=k$.
Since $G$ has no modulo $k$-orientation, it has an edge cut $T$ with $|T|\leq 3k-4$ by Lemma~\ref{lem:few}.
By \eqref{eq: local conn}, $V_0=\{v_0\}$ and $v_0$ is isolated in $G-T$, and so $d_G(v_0)\le3k-5$ since the degree of $v_0$ is even.

Since $d_{G^+}(v_0)\ge d_{G^-}(v_0)$, we can pair each negative edge
 in $\delta_G(v_0)$ with a distinct positive edge in $\delta_G(v_0)$, and pair the remaining
positive edges arbitrarily.
Let $G_1$ be the signed graph obtained from $G$ by splitting all these pairs and deleting $v_0$. Then $|E(G_1^-)|=|E(G^-)|=\epsilon(G)$.
For every partition $A,B$ of $V(G)\setminus\{v_0\}$ into nonempty sets,
$$
2e_G(A,B) =d_G(A)+d_G(B)-d_G(v_0)\ge9k-15 \ge6k-6,
$$
which implies that $d_{G_1}(A)\ge e_G(A,B)\ge3k-3$. Thus $G_1$ is $(3k-3)$-edge-connected.
By Lemma~\ref{lem:few}, $G_1$ admits a modulo $k$-orientation and hence so does $G$, a contradiction.
\end{proof}

Since $k-1$ is even, every signed graph with negativeness
at least $k-1$ is essentially $k$-unbalanced.
By Claim~\ref{cl: negativeness} and Lemma~\ref{lem:splitting-negativeness}, every signed graph $G'$ obtained from $G$ by a single splitting remains essentially $k$-unbalanced. Thus,  no choice of $V_0'$ can make $G'$ satisfy all the hypotheses of Theorem~\ref{thm: strengthened}. Otherwise, by the minimality of $G$, $G'$ would admit a modulo $k$-orientation that lifts to $G$, a contradiction.

\begin{claim}\label{cl: regular}
$G$ is a loopless $(6k-9)$-regular signed graph of order at least $3$, and $V_0=\emptyset$.
\end{claim}

\begin{proof}[Proof of Claim~\ref{cl: regular}.]
Suppose that $G$ has a negative loop at a vertex $v$.
Then $v\notin V_0$.
Splitting this loop and an edge in $\delta_G(v)$ yields a signed graph with $V_0'=V_0$ that   satisfies all the remaining hypotheses, a contradiction.
Thus $G$ is loopless.

We first show that $|V(G)|\ge 3$. Suppose $V(G)=\{u,v\}$. Then $V_0=\emptyset$.
Since $|E(G^+)|\ge |E(G^-)|=\epsilon(G)\ge k+1$ and $k$ is odd, there exist $p,q\in[0,k-1]$ such that
$$
|E(G^+)|-2p\equiv |E(G^-)|-2q\equiv0\pmod{k}.
$$
Choose $p$ positive edges and $q$ negative edges. Let $D$ be the orientation of $G$ in which the half-edges at $u$ of the chosen edges are directed toward $u$, and all other half-edges at $u$ are directed away from $u$. 
Then
$$
\begin{aligned}
\partial_D(u)&=(|E(G^+)|-2p)+(|E(G^-)|-2q)\equiv0\pmod{k},\\
\partial_D(v)&=-(|E(G^+)|-2p)+(|E(G^-)|-2q)\equiv0\pmod{k}.
\end{aligned}
$$
Hence $D$ is a modulo $k$-orientation of $G$, a contradiction. Therefore, $|V(G)|\ge3$.

We next show that $V_0=\emptyset$. Suppose that $V_0=\{v_0\}$.
Since $G$ is connected and $d_G(v_0)$ is even, we have $d_G(v_0)\ge2$ and no edge in $\delta_G(v_0)$ is a cut-edge by \eqref{eq: local conn}. If $d_G(v_0)=2$, splitting its two incident edges preserves all local edge connectivities between the other vertices.
If $d_G(v_0)\ge4$, then by Lemma~\ref{lem:Mader}, there are two edges in $\delta_G(v_0)$ whose splitting preserves all local edge connectivities between the other vertices.
In either case, the resulting graph with $V_0'=V_0$ satisfies all the remaining hypotheses, a contradiction.
Hence $V_0=\emptyset$.

It remains to show that $G$ is $(6k-9)$-regular.
Suppose that $d_G(v)\ne6k-9$ for some vertex $v$. Since $V_0=\emptyset$, $G$ is $(6k-10)$-edge-connected, and so $d_G(v)=6k-10$ or $d_G(v)\ge 6k-8$.
By Lemma~\ref{lem:Mader}, we can split two edges in $\delta_G(v)$
such that the resulting signed graph $G'$ preserves all local edge connectivities between
the other vertices.
If $d_G(v)=6k-10$, then $d_{G'}(v)=6k-12$, and we take $V_0'=\{v\}$.
If $d_G(v)\ge6k-8$, then $d_{G'}(v)\ge6k-10$, and we take $V_0'=\emptyset$.
In both cases, $G'$ admits a modulo $k$-orientation by the minimality of $G$, and hence so does $G$, a contradiction.
Therefore $G$ is $(6k-9)$-regular.
\end{proof}

Since $d_{G^+}(B)\ge d_{G^-}(B)$ for every $B\subseteq V(G)$, \eqref{eq: local conn} and Claim~\ref{cl: regular} yield that 
\begin{equation}\label{eq: degree bounds}
\begin{aligned}
d_{G^+}(B)  \ge3k-5 \quad \forall B\subsetneq V(G) \mbox{ with $|B|>0$},\\
d_{G^-}(v) \le3k-5<3k-4 \leq d_{G^+}(v) \quad \forall v\in V(G).
\end{aligned}
\end{equation}

\begin{claim}\label{cl: multiplicities}
$G^-$ has multiplicity at most $k-2$.
\end{claim}

\begin{proof}[Proof of Claim~\ref{cl: multiplicities}.]
Suppose that $e_{G^-}(u,v)\ge k-1$ for some distinct $u,v\in V(G)$. 
Choose a set $F$ of $k-1$ negative edges joining $u$ and $v$, and let $G_1$ be the signed graph obtained from $G$ by switching at $u$.
Then $F\subseteq E(G_1^+)$, and $G_1/\{u,v\}$ has at least two vertices and remains
$(6k-10)$-edge-connected.
By Lemma~\ref{le: contract}, $\epsilon(G_1/\{u,v\})\ge\epsilon(G_1)=\epsilon(G)\ge k-1$.
Thus $G_1/\{u,v\}$ is essentially $k$-unbalanced, and so admits a modulo $k$-orientation $D_1$ by
the minimality of $G$.

Let $D_2$ be the orientation of $G_1-F$ obtained by preserving the half-edge directions inherited from $D_1$ and orienting the positive edges between $u$ and $v$ outside $F$ arbitrarily.
Then $\partial_{D_2}(u)+\partial_{D_2}(v)\equiv0\pmod{k}$.
Since $k$ is odd, we can choose $a,b\in[0,k-1]$ such that
$$
a+b=k-1,
\qquad
a-b\equiv-\partial_{D_2}(u)\pmod{k}.
$$
Extend $D_2$ by orienting $a$ edges of $F$ from $u$ to $v$ and the remaining $b$ edges from $v$ to $u$. The resulting orientation is a modulo $k$-orientation of $G_1$.
Since $G$ and $G_1$ are equivalent, $G$ also admits a modulo $k$-orientation,
a contradiction.
\end{proof}

\begin{claim}\label{cl: small cut}
There is a proper subset $A\subset V(G)$ such that
\begin{equation}\label{eq: minimal cut}
\begin{aligned}
|A|\ge2,\qquad
d_{G^+}(A) & \in\{3k-5,3k-4\},\qquad d_{G^-}(A)\le3k-4,\\
d_{G^+}(B) &\ge 3k-3 \quad \forall B\subsetneq A \mbox{ with } |B|\ge 2.
\end{aligned}
\end{equation}
Moreover, $G/A$ admits a modulo $k$-orientation.
\end{claim}

\begin{proof}[Proof of Claim~\ref{cl: small cut}.]
Suppose first that $G^+$ is $(3k-3)$-edge-connected.
Since $k$ is odd, there exists $p\in [0, k-1]$ such that $2|E(G^-)|-4p\equiv0\pmod{k}$.
Since $|E(G^-)|\ge k-1$, we may choose an orientation $D_0$ of $G^-$ in which exactly $p$ negative edges  are source edges and all remaining edges are sink edges.
For each $v\in V(G)$, define $\beta(v)\in\Z_k$ by $\beta(v)\equiv-\partial_{D_0}(v)\pmod{k}$. Then $\sum_{v\in V(G)}\beta(v)=4p-2|E(G^-)|\equiv0\pmod{k}$.
Add an isolated vertex $z_0$ to $G^+$ and set $\beta(z_0)=0$.
By \eqref{eq: parity bound}, the resulting graph satisfies
the hypotheses of Lemma~\ref{thm: extension}, and hence $G^+$ admits a $\beta$-orientation.
Together with $D_0$, this yields a modulo $k$-orientation of $G$,
a contradiction.

Since $|V(G)|\ge3$, there exists a proper subset $A\subset V(G)$ with $|A|\ge2$ and $d_{G^+}(A)\le3k-4$.  
Choose such a set $A$ of minimum cardinality.
By \eqref{eq: degree bounds},
$$3k-5\le d_{G^+}(A)\le3k-4, \qquad d_{G^-}(A)\le d_{G^+}(A)\le3k-4.$$
By the minimality of $A$, we have $d_{G^+}(B)\ge3k-3$ for every
proper subset $B\subset A$ with $|B|\ge2$.
Thus, \eqref{eq: minimal cut} holds.

Clearly, $G/A$ has at least two vertices and remains $(6k-10)$-edge-connected.
By Lemma~\ref{le: contract}, we have $\epsilon(G/A)=\epsilon(G)\ge k-1$, and so $G/A$ is essentially $k$-unbalanced.
Hence $G/A$ admits a modulo $k$-orientation by the minimality of $G$.
\end{proof}

Fix a set $A$ as in Claim~\ref{cl: small cut} and a modulo
$k$-orientation $D_{G/A}$ of $G/A$.
Let $D_A$ be the orientation of $G-E(G^+[A])$ obtained by
inheriting the half-edge directions from $D_{G/A}$.
Then
\begin{equation}\label{eq: inherited boundary}
\sum_{v\in A}\partial_{D_A}(v)\equiv0\pmod{k}.
\end{equation}

Let $C=A\cap V_{3k-4}(G^+)$. By Claim~\ref{cl: regular}, we have $C=A\cap V_{3k-5}(G^-)$.
For each $v\in C$, put $b_v=e_{G^-}(v,V(G)\setminus A)$.
Define
$$
U=
\begin{cases}
C, & \text{if } |C|\le2,\\
\{u,w\}, & \text{if } |C|\ge3,
\end{cases}
$$
where $u$ and $w$ are distinct vertices of $C$ for which $b_u$ and $b_w$ are minimum and maximum, respectively.
Thus, $U\subseteq C$ and $|U|\le2$.

To extend $D_A$ to a modulo $k$-orientation of $G$, we construct
an auxiliary ordinary graph as follows:
$$
H=\left(G^+/(V(G)\setminus A)\right) \cup E_0,
$$
where $E_0$ is a set of edges joining $A$ to the contracted vertex, denoted by $z$, such that $|E_0|=3k-2-d_{G^+}(A)$ and every vertex of $U$ is incident with at least one edge in $E_0$. An illustration of $H$ is shown in Fig.~\ref{fig: auxiliary graph}.
\begin{figure}[!t]
\centering
\scriptsize
\begin{tikzpicture}[scale=0.8]
\foreach \x/\p in {0/L,9.4/R} {
  \begin{scope}[xshift=\x cm]
    \draw[black!45,line width=0.5pt]
      (0,1.5) ellipse (0.72cm and 0.85cm);
    \draw[black!45,line width=0.5pt]
      (0,-1.15) ellipse (0.72cm and 1.2cm);
    \draw[black!65,fill=black!5,line width=0.5pt]
      (-0.08,-1.57) ellipse (0.45cm and 0.66cm);

    \path (0,1.95) coordinate (\p a1);
    \path (0,1.10) coordinate (\p a2);
    \path (0,-0.55) coordinate (\p c1);
    \path (0.06,-1.30) coordinate (\p u1);
    \path (0.06,-1.85) coordinate (\p u2);

    \node[left] at (-0.78,1.5) {$A\setminus C$};
    \node[left] at (-0.78,-1.15) {$C$};
    \node at (-0.32,-1.57) {$U$};
    \node at (-0.01,1.65) {$\vdots$};
  \end{scope}
}

\draw[black!45,line width=0.5pt]
  (3.7,0) ellipse (0.65cm and 2.15cm);
\foreach \i/\y in {1/1.35,2/0.65,3/-0.10,4/-0.90}
  \path (3.55,\y) coordinate (x\i);

\foreach \v/\i in {a1/1,a2/2,c1/3,u1/4}
  \draw[line width=0.55pt] (L\v) -- (x\i);
\foreach \i in {1,...,4}
  \draw[fill=black] (x\i) circle (0.08cm);
\node at (3.55,-1.35) {$\vdots$};
\node[below] at (3.7,-2.22) {$V(G)\setminus A$};
\node at (1.6,-2.9) {$G^+$};

\draw[->,>=stealth,line width=0.7pt]
  (5.5,0) -- (6.8,0);

\path (13.1,0) coordinate (z);
\foreach \v in {a1,a2,c1,u1}
  \draw[line width=0.55pt] (R\v) -- (z);

\draw[line width=1.4pt]
  (Ru1) to[bend right=8] (z);
\draw[line width=1.4pt]
  (Ru2) to[bend right=15]
  node[midway,below=5pt] {$E_0$} (z);

\foreach \p in {L,R}
  \foreach \v in {a1,a2,c1,u1,u2}
    \draw[fill=black] (\p\v) circle (0.08cm);
\draw[fill=black] (z) circle (0.08cm);
\node[right] at (z) {$z$};
\node at (11.25,-2.9) {$H$};
\end{tikzpicture}
\caption{Illustration of the construction of $H$, with the edges in $E_0$ drawn thick.}
\label{fig: auxiliary graph}
\end{figure}
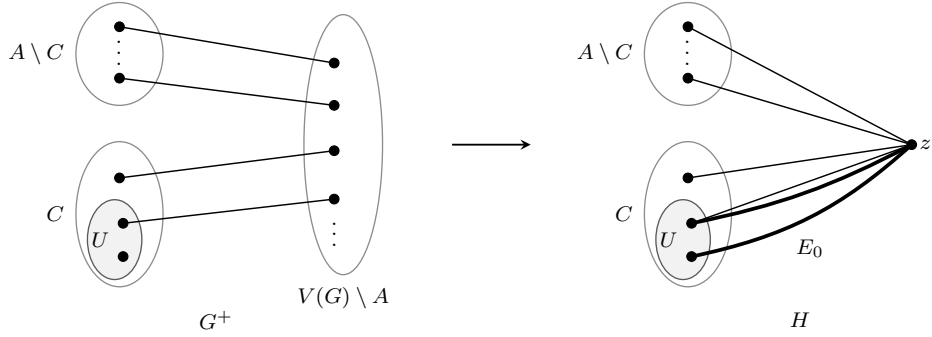
 Clearly, $|E_0|\in [2,3]$, $d_H(z)=d_{G^+}(A)+|E_0|=3k-2$ and $d_H(v)\ge d_{G^+}(v)+1=3k-3$ for $v\in U$.

Let $D_z$ be a pre-orientation of $\delta_H(z)$ in which $k-1$ edges are directed away from $z$ and the remaining $2k-1$ edges are directed toward $z$. Then
\begin{equation}\label{eq: root boundary}
\partial_{D_z}(z)=-k,
\qquad
\sum_{v\in A}\partial_{D_z}(v)=k.
\end{equation}

By \eqref{eq: degree bounds} and the construction of $H$,
every $v\in A$ with $d_H(v)\le3k-4$ belongs to
$C\setminus U$.
To apply Lemma~\ref{thm: extension}, we need the prescribed boundary to be
nonzero at these vertices. We therefore adjust $D_A$ as follows.

\begin{claim}\label{cl: partial orientation}
There exists an orientation $D_A'$ of $G-E(G^+[A])$ that
agrees with $D_A$ outside $G^-[A]$ and satisfies
\begin{equation}\label{eq: adjusted boundary}
\sum_{v\in A}\partial_{D_A'}(v)\equiv0\pmod{k}
\end{equation}
and
\begin{equation}\label{eq: boundary avoidance}
\partial_{D_A'}(v)\not\equiv\partial_{D_z}(v)\pmod{k}
\quad \forall v\in C\setminus U.
\end{equation}
\end{claim}

\begin{proof}[Proof of Claim~\ref{cl: partial orientation}.]
If $|C|\le2$, then $U=C$, and we may take $D_A'=D_A$.
Thus assume that $|C|\ge3$.
Let $s$ be the number of sink edges in $G^-[A]$ under $D_A$.
For each $v\in C\setminus U$, let $s_v$ be the number of
sink edges of $G^-[A]$ incident with $v$ under $D_A$, and define $f_v\in \Z_k$ by 
$$
f_v\equiv s_v+\frac{1}{2}(k+1)\bigl(\partial_{D_z}(v)
-\partial_{D_A}(v)\bigr)\pmod{k}.
$$
By applying Lemma~\ref{lem: star} to $G^-$ and $A$, with the vertices $u,w$ chosen above, we obtain a spanning subgraph
$F$ of $G^-[A]$ satisfying
$$
|E(F)|\equiv s\pmod{k},
\qquad
d_F(v)\not\equiv f_v\pmod{k}
\quad \forall v\in C\setminus U.
$$

Let $D_A'$ be the orientation obtained from $D_A$ by
orienting every edge of $F$ as a sink edge and every
remaining edge of $G^-[A]$ as a source edge. For each $v\in C\setminus U$, we have
$$
\begin{aligned}
\partial_{D_A'}(v)-\partial_{D_z}(v)
&=\partial_{D_A}(v)-2s_v+2d_F(v)-\partial_{D_z}(v)\\
&\equiv2\bigl(d_F(v)-f_v\bigr)\not\equiv0\pmod{k},
\end{aligned}
$$
which proves \eqref{eq: boundary avoidance}. Since the number of sink edges in $G^-[A]$ is $s$ under $D_A$ and
$|E(F)|$ under $D_A'$, by \eqref{eq: inherited boundary}, we have 
$$
\sum_{v\in A}\partial_{D_A'}(v)=\sum_{v\in A}\partial_{D_A}(v)+4\bigl(|E(F)|-s\bigr)
\equiv0\pmod{k}.
$$
This proves \eqref{eq: adjusted boundary} and completes the proof
of the claim.
\end{proof}

\noindent {\bf The final step.}  Fix an orientation $D_A'$ given by Claim~\ref{cl: partial orientation} and define
$\beta:V(H)\to\Z_k$ by
\begin{equation}\label{eq: required boundary}
\beta(z)=0,
\qquad
\beta(v)=\partial_{D_z}(v)-\partial_{D_A'}(v)
\quad \forall v\in A.
\end{equation}
By \eqref{eq: root boundary} and \eqref{eq: adjusted boundary},
$\beta$ is a $\Z_k$-boundary.
Moreover, by \eqref{eq: boundary avoidance}, $\beta(v)\ne0$ for every $v\in C\setminus U$.

We now verify that $H$ together with $D_z$ and $\beta$  satisfies the hypotheses
of Lemma~\ref{thm: extension}.
Clearly, $|V(H)|=|A|+1\ge3$ and
$\partial_{D_z}(z)=-k\equiv\beta(z)\pmod{k}$.
Since $d_H(z)=3k-2$ is odd and $\beta(z)=0$,
we have $|\tau(z)|=k$, and hence
$$
d_H(z)=2k-2+|\tau(z)|.
$$
Now let $B$ be a nonempty proper subset of $A$.
If $d_H(B)\ge3k-3$, then by \eqref{eq: parity bound}, 
$$d_H(B)\ge2k-2+|\tau(B)|.$$ 
Suppose now that $d_H(B)\le3k-4$. By \eqref{eq: minimal cut} and the construction of $H$, we must have $B=\{v\}$ for some $v\in A$. Moreover, $v\in C$ since $
3k-4\le d_{G^+}(v)\le d_H(v)\le3k-4$ by \eqref{eq: degree bounds},
Note that every vertex of $U$ is incident with at least one edge of $E_0$ and therefore has degree at least $3k-3$ in $H$.
Thus $v\in C\setminus U$, and by \eqref{eq: boundary avoidance},
$
\beta(v)=\partial_{D_z}(v)-\partial_{D_A'}(v)
\not\equiv0\pmod{k}.
$
Since $d_H(v)=3k-4$ is odd, it follows that
$|\tau(v)|\le k-2$. Hence
$$
d_H(v)=3k-4\ge2k-2+|\tau(v)|.
$$

By Lemma~\ref{thm: extension}, $D_z$ extends to a
$\beta$-orientation of $H$.
Together with $D_A'$, the restriction of this
$\beta$-orientation to $G^+[A]$ yields an orientation $D$ of $G$. By \eqref{eq: required boundary}, we have
$$
\begin{aligned}
\partial_D(v)
&\equiv\partial_{D_A'}(v)+\beta(v)-\partial_{D_z}(v)
\equiv0\pmod{k}
&& \forall v\in A,\\
\partial_D(v)
&\equiv\partial_{D_A'}(v)\equiv\partial_{D_A}(v)
\equiv\partial_{D_{G/A}}(v)\equiv0\pmod{k}
&& \forall v\in V(G)\setminus A.
\end{aligned}
$$
Thus $D$ is a modulo $k$-orientation of $G$, a contradiction. This completes the proof of Theorem~\ref{thm: strengthened}.
\end{proof}


\end{document}